\documentclass[a4paper,10pt]{article}
\usepackage{QED}
\usepackage{amsmath}
\usepackage{amssymb}
\usepackage{a4}
\def\a{\alpha}
\def\b{\beta}

\def\G{\Gamma}

\def\l{\lambda}

\def\ep{\epsilon}
\def\m{\mu}

\def\o{\omega}
\def\ep{\varepsilon}
\def\f{\rightarrow}
\def\trrd{\triangleright}

\def\v{\vdash}

\def\ou{\vee}
\def\etrd{\wedge}
\def\<{\langle}
\def\>{\rangle}
\def\F{\displaystyle\frac}

\newtheorem{thm}{Theorem}

\newtheorem{lem}[thm]{Lemma}

\newtheorem{defn}[thm]{Definition}

 \newcommand{\N}{ \mathbb{N}}

\newcommand{\be}{=_{\beta \eta}}

\newcommand{\ar}{\hspace*{0.1mm}   \rightarrow }
\newcommand{\LA}{ $\lambda$ }

\newcommand{\NA}{ \overline{\mathbb{N}}}

\def\bt{\begin}
\def\et{\end}

\newcommand{\di}{\#}

\newcommand{\comment}[1]{}

\begin{document}

\title{Counting proofs in propositional logic}
\author{ Ren\'e DAVID\footnote{Lama, Universit\'e de
Savoie, Campus Scientifique. 73376 Le Bourget du lac. Email : rene.david@univ-savoie.fr} \ \ \ Marek
ZAIONC\footnote{ Theoretical Computer Science, Jagiellonian University, Lojasiewicza 6, 30-348
Krak\'{o}w, Poland. Email : zaionc@tcs.uj.edu.pl. Research described in this paper is supported by
Polish Ministry of Science and Higher Education grant  NN206 356236 }}
\date{\today}
\maketitle

\begin{abstract}
We give a procedure for counting the number of different proofs of
a formula in various sorts of propositional logic. This number is
either an integer (that may be 0 if the formula is not provable)
or infinite.
\end{abstract}

\section{Introduction}
The aim of the paper is to give a procedure for counting the
number of different normal proofs of a formula in propositional
logic. By the well known Curry Howard correspondence, this is
similar to counting the number of different normal closed terms of
some fixed type in an extension of the $\lambda \mu$ calculus.

We show that this number is the least fix-point of a system of
polynomial equations in some natural complete lattice and we give
an algorithm for finding such a least fix-point.

The similar problem of counting closed typed lambda terms was
studied (see \cite{BeY79}) but never published by Ben- Yelles.
Some description of the Ben-Yelles solution can be found in
Hindley's book \cite{Y79}. Similarly Hirokawa in \cite {H91}
proved that the complexity of the question whether a given simple
type (implicational formula) possess an infinite number of normal
terms (or infinite number of proofs) is polynomial space complete.
Recently similar research about counting $\l$-calculus objects for
program synthesis was done by Wells and Yakobowski in
\cite{WY2005}.

 \section{The logic  }
 \subsection{Formulae and proofs}

\begin{defn}
Let ${\cal A}$ be a set (possibly infinite) of atomic constants.
 The set ${\cal F}$ of formulae is defined by the following grammar

$${\cal F}  ::= \; {\cal A}\cup \{\bot\} \; \mid \;   {\cal F} \f
{\cal F} \; \mid \; {\cal F} \etrd {\cal F} \; \mid \; {\cal F}
\ou {\cal F} $$
\end{defn}

\noindent We assume that $\bot \not\in {\cal A}$ and, as usual,
$\neg F$ will be an abbreviation for $F \f \bot$.

\begin{defn}\label{pm}

The rules for proofs in classical logic are the following.

\begin{center}

  $\F{}{\G , A \v  A} \, ax$

\medskip

 $\F{\G, A \v  B} {\G \v
 A \f B} \, \f_i$ \hspace{0.5cm}
 $\F{\G_1 \v  A \f B \quad \G_2 \v  A} {\G_1,\G_2 \v  B }\, \f_e$

\medskip

$\F{\G_1 \v  A_1 \quad \G_2 \v  A_2} {\G_1 , \G_2 \v A_1 \etrd
A_2} \, \etrd_i$ \hspace{0.5cm} $\F{\G \v  A_1 \etrd A_2} {\G \v
A_i } \, \etrd_e$

\medskip

 $\F{\G \v
A_j} {\G \v  A_1 \ou A_2} \, \ou_i$ \hspace{0.5cm}$\F{\G \v  A_1
\ou A_2 \quad \G_1 ,  A_1 \v  C \quad \G_2 , A_2 \v
 C} {\G,\G_1,\G_2 \v  C} \, \ou_e$

\medskip

 $\F{\G ,  \neg A \v  \bot} {\G \v A} \bot_e$ \hspace{0.5cm}
 $\F{\G ,  \neg A \v A} {\G ,  \neg A\v
\bot} \bot_i$
\end{center}

\end{defn}

\subsection{Terms coding proofs}

It is well known that a proof, in intuitionistic implicational logic, can be coded by
a simply typed $\l$-term. The same thing can, in fact, be done for proofs, in
classical logic, of any kind of formulae. The extension from intuitionistic logic to
classical logic is the $\l\m$-calculus introduced by Parigot in \cite{Par92}. The
extension to formulae using all the usual connectors has been introduced by
 de Groote in \cite{DGro01}. The next definition is a  presentation of this calculus.

\begin{defn}\label{lm}
Let ${\cal V}$ and ${\cal W}$ be disjoint sets of variables. The
set of $\l\m^{\f\etrd\ou}$-terms is defined by the following
grammar

$${\cal T} ::= {\cal V} \ | \ \l {\cal V}. {\cal T} \ | \ ({\cal T} \; {\cal E}) \ | \ \<{\cal T} , {\cal T} \> \ | \ \o_1
{\cal T} \ | \ \o_2 {\cal T} \ | \ \mu {\cal W}. {\cal T} \ | \
({\cal W} \; {\cal T})$$

$${\cal E} ::=  {\cal T} \ | \  \pi_1 \ | \ \pi_2 \ | \ [{\cal V}.{\cal T} ,{\cal V}.{\cal T}]$$

\end{defn}

The next definition shows how the terms introduced in definition
\ref{lm} code the proofs.

\begin{defn}\label{tr}
The typing rules for the $\l\m^{\f\etrd\ou}$-terms are as follows
\begin{center}

$\F{}{\G , x : A \v x : A} \, ax$ \hspace{0.5cm} $\F{\G, x: A \v M
: B} {\G \v \l
x.M : A \f B} \, \f_i$\\[0.5cm]

\medskip

$\F{\G_1 \v M : A \f B \quad \G_2 \v N : A} {\G_1,\G_2 \v (M \; N)
: B }\, \f_e$

\medskip

$\F{\G , \a : \neg A \v M : A} {\G , \a : \neg A\v (\a \; M) :
\bot} \bot_i$\hspace{0.5cm} $\F{\G , \a : \neg A \v M : \bot} {\G
\v \mu \a.M : A} \bot_e$

\medskip

$\F{\G_1 \v M : A_1 \quad \G_2 \v N : A_2} {\G_1 , \G_2 \v \< M,N
\> : A_1 \etrd A_2} \, \etrd_i$ \hspace{0.5cm} $\F{\G \v M : A_1
\etrd A_2} {\G \v  (M \; \pi_i): A_i } \, \etrd_e$

\medskip

 $\F{\G \v M :
A_j} {\G \v \o_j M : A_1 \ou A_2} \, \ou_i$

\medskip

$\F{\G \v M : A_1 \ou A_2 \quad \G_1 , x_1 : A_1 \v N_1 : C \quad
\G_2 , x_2 : A_2 \v N_2 : C} {\G,\G_1,\G_2 \v (M\;[x_1.N_1,
x_2.N_2]) : C} \, \ou_e$
\medskip
\end{center}
\end{defn}

\noindent {\bf Remark}

Note that, in definition \ref{pm}, the letter $\G$ represents a
finite multi-set of formulae whereas, in definition \ref{tr}, it
represents a finite multi-set of {\it indexed} formulae i.e. a
finite set of pairs  denoted as $x:A$ or $\a : \neg A$ where $x
\in {\cal V}$, $\a \in {\cal W}$ and
 $A \in {\cal F}$ (where each variable  occurs
only once).

In the rest of the paper, we will continue to use the same
notation for these two formally distinct notions. Such a multi-set
will be called a {\it context}.  In a particular sentence which of
the two notions is meant will usually be clear ... from the
context.

\begin{defn}
The set $G$ of goals is the set of ordered pairs denoted as  $\G
\v  A$  where
  $A \in {\cal F}$  and  $\G$ is a context.
\end{defn}

\subsection{Normal terms and proofs}

To avoid to have, for each formula, either zero or infinitely many
proofs, we  only consider  proofs satisfying two conditions.

\begin{enumerate}
  \item The first one is usual : we only look at normal proofs i.e. proofs
with no cuts i.e. proofs such that the term that represents it is
normal i.e. cannot be reduced by the reduction rules of definition
\ref{redex} below which corresponds  to the usual notion of cut
elimination in natural deduction. Since every term is normalizing
i.e. can be reduced to a normal term (cf. theorem \ref{sub}, item
1), if a formula has a proof then it also has a normal proof. Thus
the restriction does not change the problem.
  \item The second restriction, though quite natural,  is less usual but
also necessary to avoid to have, for each formula, either zero or
infinitely many proofs.  It is as follows.

\begin{enumerate}
  \item When we are in one of the branch
   of a proof by case (we have $A \vee B$ as an hypothesis and we assume, for example, $A$), we are no more allowed to, again,
   distinguish the same two cases i.e. we delete the hypothesis $A \vee
   B$.
  \item We forbid to prove $\bot$ or $\neg C$ by
   contradiction.
  \item When  we are in a part of the proof in which
   we already have assumed $\neg B$, toward a contradiction,  we
   are no more allowed to prove $B$ by contradiction.
\end{enumerate}

\noindent A proof satisfying these three conditions will be called
{\it fair}.  It is  easy to check that if a formula is provable
then it has a fair (normal) proof and thus asking for fairness
does not change the problem. Since fairness is less usual, we will
mention in the comments of section \ref{eq} where it appears in
the research for proof.

\end{enumerate}

Note finally that we may want to add some (optional) restrictions
to the number of proofs we are looking for. They will be given in
section \ref{rest}.

\begin{defn}\label{redex}

The reduction rules for the $\l\m^{\f\etrd\ou}$-calculus are given
below.
 Variables $M,N,L$ are in ${\cal T}$ and  $\ep$ is in
 ${\cal E}$.
 A variable $x$ belongs to ${\cal V}$ while $\a$ is taken from  ${\cal W}$.

 \begin{center}

  $(\l x. M \; N) \trrd_{\b} M[x:=N]$

 \medskip

$(\< M_1,M_2 \> \; \pi_i) \trrd M_i$

 \medskip

 $(\o_i M\; [x_1. N_1 ,x_2. N_2]) \trrd N_i[x_i:=M]$

  \medskip

 $(M \; [x_1.N_1 , x_2.N_2] \; \ep)
\trrd (M \; [x_1.(N_1 \; \ep) , x_2.(N_2 \; \ep)])$

 \medskip

$(\mu \a. M \; \ep) \trrd_{\m} \mu \a. M[(\a \; L):= (\a \; (L \;
\ep))]$
\end{center}
\end{defn}

\noindent {\bf Remarks}

 - The first three rules correspond to the elimination of a
logical cut: an introduction rule is immediately followed by the
corresponding elimination rule.

- The fourth rule corresponds to the so-called {\it permutative
conversion}: if a case analysis is followed by an elimination rule
the elimination can be done before the case analysis.

- The last rule corresponds to the so-called {\it classical cuts}

\noindent Note that the two last rules  are necessary to ensure
that a normal proof has the sub-formula property (cf. theorem
\ref{sub}, item 2).

\begin{defn}
Let $t$ be a $\l\m^{\f\etrd\ou}$-term and $g = \G \v A$ be a
  goal.
\begin{enumerate}
  \item  We say that $t$ is a proof of $g$ if $\G \v t :A$.
    \item We say that $t$ is  normal if it contains no redex i.e.
if it cannot be reduced by the rules of definition \ref{redex}.
\end{enumerate}

\end{defn}

\begin{thm}\label{sub}
Let $t$ be a proof of  $g=\G \v A$. Then,

\begin{enumerate}
  \item $t$ can be reduced into a normal proof of $g$.
  \item If $t$ is normal and $B$ is a formula that occurs in the
  typing tree of $t$ then, there is a
sub-formula $C$ of a formula in $\{A\} \cup \G$ such that $B=C$ or
$B = \neg C$.
\end{enumerate}
\end{thm}
\begin{proof}
Standard. See for example \cite{Tro} or \cite{Vand}.
\end{proof}

\begin{thm}\label{thm}
There is an algorithm that, given a formula $F$, computes the
number (i.e. either an integer or $\infty$) of distinct normal and
fair proofs of $F$.
\end{thm}
\begin{proof}
This is an immediate corollary of theorem \ref{main} below whose
statement and  proof is given in the next section.
\end{proof}

\section{Proof of the main result}

\subsection{The idea of the proof}
The idea of the proof is quite simple. To each goal $g$ of the
form  $\G \v A$ we associate a variable $n_g$  that, intuitively,
gives the number of normal and fair proofs of $g$. By looking at
the possible ways of proving $g$ (either use an introduction rule
or an elimination rule or a proof by contradiction) we get
equations relating the $n_g$. We will show that the number we are
looking for is the minimal solution of this set of equations. The
two main technical difficulties are the following.

- We have to be able to compute the solution of these equations. This follows from the fact that they
 only use integers, addition
 and
 multiplication. An addition corresponds to the possibility of proving
 a goal in different ways. A multiplication means that, to prove the goal, we have to prove two
 different  things. Thus the equations are polynomial and we will show that, for this kind of
 equations, we can always compute the minimal solution.

- The other point is a bit more difficult. To be able to compute
its solution, the set of equations must be finite but, without
sufficient care, it is not ! Since, by the sub-formula property
(theorem \ref{sub} above), we know that the formulae that appear
in a normal proof are sub-formulae of the initial formula, the set
of goals must, intuitively, be finite (which would imply that the
set of equations also is finite) but since, in $\G$, a formula can
be repeated many times it is not true that the set of goals is
finite. To solve this problem, we proceed as follows. When, in a
proof of some goal we introduce a new goal, say $h$, which is the
same as a goal $h'$ that has already been introduced except that
it adds some hypothesis that were already present in $h'$, we do
not consider it as a  new one i.e. we do not build an equation for
it. This is because we can show that $h,h'$ have the same number
of proofs. But, to do that, we need some book keeping because to
show that $h,h'$ have the same number of normal and fair proofs,
we need the fact that $h$ and $h'$ are, somehow, in the same part
of a proof. This will be ensured by the order we put on the
variables $n_g$. Doing in this way, Konig's lemma ensures that the
set of equations is finite.

\subsection{Polynomials}

\begin{defn}
\begin{enumerate}
  \item The set $\N \cup \{ \omega \}$ will be denoted as
$\NA$
  \item The usual order and operations on $\N$ are extended to
  $\NA$ by
\begin{itemize}
  \item  $i \leq \omega$ and $i + \omega = \omega + i = \omega$ for every $i \in \NA$,
  \item $0 \cdot \omega = \omega \cdot 0 = 0$,
  \item
 $i \cdot \omega = \omega \cdot i = \omega$ for every $i  \neq 0$.

\end{itemize}

  \item The set $\NA^k $ is
 naturally ordered by $(a_1 ,...,a_k) \leq (b_1 ,...,b_k )$
 if $a_i \leq b_i$ for all $i$.
\end{enumerate}

\end{defn}

\begin{lem} $\NA^k$ is a complete lattice.
\end{lem}\label{completelattice}
\begin{proof}
Obvious.
\end{proof}

\bt{defn}\label{polynomial}
\begin{enumerate}
  \item The set of polynomials is the least set of functions (of several
variables) from $\NA$ to $\NA$ that contains the constant
functions and is closed by addition and multiplication.
 \item The order on polynomials is the point-wise order, i.e. if $f(x_1, ...,
  x_n),g(x_1, ...,
  x_n)$ are polynomials,
  $f \leq g$ iff $\forall x_1, ...,
  x_n,
f(x_1, ...,
  x_n) \leq g(x_1, ...,
  x_n)$.
\end{enumerate}
 \et{defn}

\begin{defn}
\begin{enumerate}
  \item  A polynomial system of equations ($PSE$ for short) is a
  set $\{E_1, ..., E_n\}$  where $E_i$ is  the equation $x_i=f_i(x_1, ...,
  x_n)$ and $f_i$ is a polynomial in the variables $x_1, ...,
  x_n$.
Such a system will be abbreviated as $\Vec{x}=F(\Vec{x})$.
  \item Let $\Vec{x}=F(\Vec{x})$ by a $PSE$. We say that $\Vec{a}$
   is a minimal solution  of
the system if $\Vec{a} = F(\Vec{a})$ and, for every $\Vec{b}$ such
that $\Vec{b}=F(\Vec{b})$, we have $\Vec{a} \leq \Vec{b}$.
\item We denote by $F^i$ the $i$-iteration of $F$, i.e.
$F^0(\Vec{x})=\Vec{x}$ and $F^{i+1}(\Vec{x})=F(F^i(\Vec{x}))$.
\end{enumerate}
\end{defn}

\bt{prop}\label{tarski}
 Let $\Vec{x}=F(\Vec{x})$ be a $PSE$. Then,
 this system has a (unique) minimal solution $\Vec{a}$
 (that we will denote  by $min(F)$). Moreover
 we have $min(F)=\bigsqcup_{i=0}^{\infty} F^i (
 \Vec{0})=\bigcap \{\Vec{x} \ | \ F(\Vec{x}) \leq \Vec{x} \}$.

\end{prop}
\begin{proof}
Since it is easy to check that $F$ is increasing, this is a
special  case  of the Knaster-Tarski lemma.
\end{proof}

\begin{lem}\label{prepa}
Let $f(x,\Vec{y}) = f_0
 (\Vec{y})+ \sum_{i\geq 1} f_i
(\Vec{y}) x^i$ be a polynomial (where  $\Vec{y}$ is possibly
empty) and let $h(\Vec{y})=\sum_{i\geq 1} f_i (\Vec{y})$. Then,
$g(\Vec{y})=f_0 (\Vec{y}) + f_0 (\Vec{y}) \cdot h (\Vec{y}) \cdot
\omega$ is the minimal solution of the equation $x=f(x,\Vec{y})$.
\end{lem}
\begin{proof}
 If $f_0 (\Vec{y} ) =0 $ then the minimal solution is $0$.
 If  $h (\Vec{y}) = 0$,
 then for all $i \geq 1$, $f_i (\Vec{y}) =0$ and the minimal solution
is $f_0 (\Vec{y})$. Otherwise, it is easy to check that the
minimal solution is  $\omega$. In all cases the minimal solution
is $g(\Vec{y})$.
\end{proof}

\begin{lem}\label{compute}
Let $\Vec{x}=F(\Vec{x})$ by a $PSE$. The minimal solution of this
system  can be computed from $F$.
\end{lem}
\begin{proof}
The algorithm to compute this solution is the following. Choose
one variable, call it $x$ and call $\Vec{y}$ the remaining
variables. The system then looks like:  $x= f(x,\Vec{y})$ and
$\Vec{y}=G(x,\Vec{y})$. Use lemma \ref{prepa} to find the
polynomial $g(\Vec{y})$ which is the minimal solution of the
equation $x= f(x,\Vec{y})$. Repeat the process with the system
$\Vec{y}=G(g(\Vec{y}),\Vec{y})$. It is clear that, in this way, we
find a solution of the system. Denote by $(a,\Vec{b})$ this
solution.  By proposition \ref{tarski}, let $(x_0,
\Vec{y_0})=min(F)$. Since $(a,\Vec{b})$ is a solution of the
system we have $(x_0, \Vec{y_0})\leq (a,\Vec{b})$. Thus it remains
to show that $(a,\Vec{b})\leq (x_0, \Vec{y_0})$. Since $x_0$ is a
solution of the equation
 $x = f(x, \Vec{y_0} )$ we have $g( \Vec{y_0} ) \leq x_0$. Define $F'$ by
  $F'(\Vec{y})= G(g(\Vec{y}), \Vec{y})$. By
 the monotonicity of $G$,
 $F' ({\Vec{y}}_0 )= G(g( {\Vec{y}}_0 ),
 {\Vec{y}}_0 ) \leq G(x_0 ,{\Vec{y}}_0 ) = {\Vec{y}}_0 $.
 But since the minimal solution of  $F'$ is
 $\bigcap \{\Vec{y} \ | \ F'(\Vec{y}) \leq \Vec{y} \}$ we have $\vec{b} \leq \Vec{y_0}$. By the monotonicity of $g$,
 $a = g( b_0 ) \leq g( {\Vec{y}}_0 ) \leq
 x_0$.
\end{proof}

\subsection{Some preliminary results}

\begin{defn}
\begin{enumerate}
  \item We will denote by ${\cal F'}$ the set of formulae to which we
  have added a special element denoted as $*$.

  \item Let $E$ be a set of lists of elements of  ${\cal F'}$ and $A$ be a formula. We will
denote by $[A :: E]$  the set $\{[A:: L] \ | \ L \in E\}$ where
$[A:: L]$ denotes the list $L$ on the beginning of which we have
added $A$.
\end{enumerate}
\end{defn}

\noindent {\bf Remark}

Note that the definition implies that, if $E$ is empty, then so is
$[A :: E]$.

\begin{defn}
Let $A,B$ be formulae. The set $Elim (A,B)$  of lists of elements
of ${\cal F'}$ is defined, by  induction on the size of $A$, in
the following way.
\begin{enumerate}

  \item If $A=B$, then $Elim (A,B)={[*]}$.

  \item If $A \neq B$ then,

  - If $A$ is atomic,  $Elim (A,B)=\emptyset$

   - If $A=C \f D$,  $Elim (A,B)=[C :: Elim(D,B)]$

   - If $A= A_1 \etrd A_2$,  $Elim (A,B)= Elim(A_1,B) \cup
  Elim(A_2,B)$

  - If $A= A_1 \ou A_2$, $Elim (A,B)= \{[A]\}$
\end{enumerate}
\end{defn}

\begin{lem}
Let $A,B$ be formulae and let $L \in Elim(A,B)$. Then the last
element of $L$ is either $*$ or a disjunction.
\end{lem}
\begin{proof}
By induction on $A$.
\end{proof}

\noindent {\bf Comments and examples}

\begin{enumerate}
  \item The role of the particular symbol $*$ and the set $Elim(A,B)$ will
become clear in item 3 of the next lemma. The intuition is the
following. $Elim(A,B)$ is the set of lists $L$ satisfying the
following properties.

- If $L=[A_1 :: ... :: A_n :: *]$  then, to be able to prove $B$
 in some context $\G$ {\em by using a sequence of elimination rules starting with  $A$}, it is enough to
 prove $A_1, ..., A_n$ in the context $\G$.

 - If $L=[A_1 :: ... :: A_{n-1} :: D_1 \ou D_2]$ then, to be able to
prove $B$
 in some context $\G$ {\em by using a sequence of elimination rules starting with  $A$}, it is enough to
 prove $A_1, ..., A_{n-1}$ in the context $\G$ and to prove $B$ both in the contexts
  $\G \cup \{D_1\}$ and $\G \cup \{D_2\}$.
  \item Assume $B, B'$ are distinct atomic formulae and $A=(A_1 \f D_1 \ou
  D_2) \etrd (A_2 \f A_3 \f B) \etrd (A_4 \f B')$.
  Then $Elim(A,B)=\{L_1,L_2\}$ where $L_1=[A_1 :: D_1 \ou
  D_2]$ and $L_2=[A_2 :: A_3 :: *]$

\end{enumerate}

\begin{lem}\label{normal}

 Let $t$ be a normal proof of $\G \v B$. Then, $t$ is in one of
 the following form (where the $t_i$ are normal)
  \begin{enumerate}
  \item Either

  - $t=\l x. t_1$, $B=B_1 \f  B_2$ and $\G, x :B_1 \v t_1 : B_2$

  -  $t=\m \a. t_1$ and $\G, \a : \neg B \v t_1 : \bot$

  -  $t=\< t_1,t_2
  \>$, $B=B_1 \etrd B_2$
  and $\G \v t_i : B_i$

  -  $t=\o_i t_1$, $B=B_1 \ou B_2$ and $\G \v t_1 : B_i$.

  \item Or  $t=(\a \ t_1)$ and $\G \v t_1 : A$ where $\G \v \a :
  \neg A$
  \item Or  $t=(x\ t_1 \ ... \ t_n)$ and, for some $A$ such that
  $\G \v x:A$ and some $L \in Elim(A,B)$, we have

  - either $L=[A_1 :: ... :: A_n :: *]$ and the $t_i$ are proofs of $\G \v A_i$

   - or $L=[A_1 :: ... : A_{n-1} :: D_1 \ou D_2]$ and, for $i <n$,  the $t_i$ are
    proofs of $A_i$ and $t_n=[x_1.u_1 , x_2.u_2]$ and the
   $u_i$ are  proofs of $\G, x_i : D_i  \v B$.
  \end{enumerate}
\end{lem}

\begin{proof}
By induction on the size of the proof. The only non immediate
point is that we cannot use an elimination rule when the type  is
a disjunction. This is  because, otherwise, we will get a proof of
the form $(x \ t_1 \ ... \ t_k \ [x_1.N_1 , x_2.N_2] \ \ep)$ which
is not normal.
\end{proof}

\begin{defn}
Let $t$ be a normal proof. The size of $t$ (denoted as $size(t)$)
is defined as follows.
\begin{enumerate}
  \item $size(\l x. t_1)=size(\m \a. t_1)=size(\o_i t_1)=size(t_1)+1$
  \item $size(\< t_1,t_2\>)=max(size(t_1), size(t_2))+1$
  \item $size((x \ t_1 \ ... \ t_n)=max(size(t_1), ...,
  size(t_n))+1$
\end{enumerate}
\end{defn}

\begin{defn}
\begin{enumerate}
  \item The set ${\cal P}$ of partial (normal) terms is defined by the
following grammar
$${\cal P}:= {\cal V} \ | \ G  \ | \ \l x. {\cal P} \ | \  \m\a. {\cal P} \ | \  \<{\cal P}, {\cal P}\> \ | \
\o_i {\cal P} \ | \  (x \ {\cal P} \ ... \ {\cal P})$$
  \item The typing rules for ${\cal P}$ are the ones of ${\cal T}$
  plus the additional rule

  \begin{center}
$\F{} {\G \v g:A} $ \hspace{1cm} if $g=\G \v A$
\end{center}

\end{enumerate}
\end{defn}

\noindent {\bf Remark}

A normal proof is partial term that contains no goal.

\begin{defn}
 Let $g$ be a  goal. We denote
by $\di(g)$ the number (considered as an element of $\NA$) of
distinct normal and fair proofs of $g$.
\end{defn}

\begin{defn}
\begin{enumerate}
  \item Let $\G, \G'$ be two contexts. We say that $\G$ is equivalent to $\G'$
  (denoted as $\G \sim \G'$) if, for any $A \in {\cal F}$, $\G$
  contains a declaration $x : A$ iff $\G'$ contains a declaration $y :
  A$.
  \item Let  $g= \G \v B$ and $g'=\G'\v B'$.  We say that $g$ is equivalent to $g'$
  (denoted as $g \sim g'$) if $B=B'$ and  $\G \sim \G'$.
\end{enumerate}
\end{defn}

Thus two goals $g,g'$ are equivalent iff their conclusions are the
same and they have same set of hypothesis but  each hypothesis may
appear a different number of times in $g$ and $g'$.

\begin{lem}\label{27}
Let $t$ be a partial proof of goal $g$. Assume $t\neq g$ and
contains some goal $g' \sim g$.   Then $\di(g)=\di(g')$.
\end{lem}
\begin{proof}
It is clear that $g$ has no proof iff $g'$ has no proof. Assume
then that  $\di(g) \geq 1$. Let $g''=\G'' \v A \sim g$ be such
that, for any formula $B$, the number of occurrences of $B$ in
$\G$ or in $\G'$ is less or equal to the number of occurrences of
$B$ in $\G''$.

We first show that $\di(g'')= \o$. It is clear that the term $t'$
obtained from $t$ by replacing $g'$ by $g''$ also is a partial
proof of $g''$ and, if $u$ is a proof of $g$, it also is a proof
of $g''$. Then,  the $u_n$  defined by $u_0=u$ and
$u_{n+1}=t'[g'':=u_n]$  are distinct normal and fair proofs of
$g$.

We then show that $\di(g)= \o$ (and, by symmetry, $\di(g')= \o$).
Assume, toward a contradiction, that $\di(g)$ is finite. To each
proof of $g''$ associate the proof of $g$ obtained by replacing
the occurrences of a variable in $\G'' - \G$ by one with the same
type in $\G$. Since $\di(g)$ is finite and $\di(g'')$ is infinite,
there are infinitely many proofs of $g''$ that have the same image
by this transformation. But this is impossible since, in a proof,
each variable occurs only finitely many times.
\end{proof}

\subsection{The equations}\label{eq}

To  every goal $g=\G \v A$ we associate a polynomial system of
equations (denoted as $PSE(g)$) of the form $\Vec{n}=P(\Vec{n})$
where a goal $g_i$ is associated to each variable $n_i$ and $p_i$
is a polynomial that, intuitively,  computes  the number of normal
 and fair proofs of $g_i$ of a given size from the number of
proofs (of smaller size) of the other goals needed to prove $g_i$.

 $PSE(g)$ is defined by the following algorithm.  This algorithm builds, step by step,  a partially ordered
   set $V$ of variables (denoted as $n$ with some index), a function $F$ that associates goals to the variables and a set $E$
 of
 equations of the form $n_i=p_i(\Vec{n})$. We will show (see lemma \ref{fin} below)
 that it terminates. $PSE(g)$ will be the
 set of equations we have built when the algorithm terminates.

It is important to note that the function $F$ is not necessarily
injective i.e. to different variables may correspond to the same
goal. The reason will be given in the comments after the
description of the algorithm.

\medskip

\noindent {\em - Initial step}
\medskip

Set $V=\{n_0\}$, $F(n_0)=g$ and $E=\emptyset$.

\medskip

\noindent {\em - General step}

 If, for all $n_i \in V$, there is
an equation
   $n_i=p_i (\Vec{n})$ in $E$, then stop. Otherwise, choose some $n_i$ for which $E$ has no equation.
 We introduce new variables and build the polynomial $p_i$ as  the sum of  three
 polynomials in the following way. The first one corresponds to a proof of
 $F(n_i) = \G \vdash B$
 beginning by an introduction rule,
 the second corresponds to a proof of $F(n_i)$ by
 contradiction and the last corresponds to a proof of $F(n_i)$ by using some hypothesis and several
 elimination rules.

\medskip

In the definition of these polynomials we will adopt the following
convention. If $h$ is a goal, when we say ``  let $n$ be a
variable for $h$ '' (we will also say ``   $n$ is a name for $h$
'') this will mean that either $F(n_j) \sim h$ for some $n_j <
n_i$ and then $n$ is such an $n_j$ (if there are several choose
  one) or, if no such variable exists, choose a fresh index $j$ and set
  $F(n_j)=h$.
  For each variable $n_j$ introduced in this way,
we set $n_j > n_k$ for each $k$ such that $n_i \geq n_k$.

\begin{enumerate}
  \item The first polynomial $P$ depends on the main connector of $B$.

\begin{enumerate}
\item If $B$ is an atomic formula, then $P=0$
  \item If $B=C \f D$ then let
$h= \G, y:C \v D$, then let $P=n_j$ where $n_j$ is a variable for
$h$.

  \item If $B=B_1 \etrd B_2$. Let $h_i$ be the goal $\G \v B_i$.
  Then $P=n_{i_1}. n_{i_2}$ where $n_{i_1}, n_{i_2}$ are variables for $h_1,h_2$.
  \item If $B=B_1 \ou B_2$. Let $h_i$ be the goal $\G \v B_i$.
  Then $P=n_{i_1} + n_{i_2}$ where $n_{i_1}, n_{i_2}$ are variables for $h_1,h_2$.

\end{enumerate}

  \item The second polynomial $Q$ is as follows.

\begin{enumerate}
  \item If $B = \bot$ or $B=\neg C$ or if there is already in $\G$ an hypothesis of the form $\a : \neg B$,
  then $Q=0$.
  \item Otherwise, let $h=\G, \a : \neg B
  \v \bot$ and $Q=n_j$ where $n_j$ is a variable for
$h$.

\end{enumerate}

\item The last polynomial is the sum of (over all the hypothesis
$H$ in $\G$) of the polynomials $R_H$ defined as follows.
\begin{enumerate}
  \item If $H$ is $x : A$, $R_H$ is the sum (over $L \in Elim(A, B)$)
of the polynomials $R_{H,L}$ defined below.

- Assume $L=[A_1 ::  ... :: A_p :: *]$. Then $R_{H,L}= n_{i_1}.\
... \ .n_{i_p}  $ where $g_i= \G \v A_i$ and $n_{i_1}, ...,
n_{i_p}$ are variables for $g_1, ..., g_p$. In particular, if
$p=0$, this means $R_{H,L}=1$.

- Assume $L=[A_1 ::  ... :: A_p :: D_1 \ou D_2]$. Then, let $g_i=
\G \v A_i$,  $h_i= \G', y:D_i \v B$ where $\G'$ is $\G$ from which
we have deleted the hypothesis $x : A$. Let $n_{i_1},  ... ,
n_{i_p}$ be variables for $g_1, ..., g_p$, let $n_{j_1}, n_{j_2}$
be variables for $h_1,h_2$. Then $R_{H,L}= n_{i_1}.\ ... \
.n_{i_p}. n_{j_1} . n_{j_2}$
  \item If $H$ is $\a : \neg A$ then $R_H=n_j$ where $h=\G \v A$ and $n_j$ is a variable for $h$.

\end{enumerate}

\end{enumerate}

  \medskip

\noindent   {\bf Comments}

\begin{enumerate}
  \item

 Eliminating the hypothesis $x : A$ in the case
   of an elimination of the disjunction is condition (a) of
   fairness.
The fact that $Q=0$ in the first case of a proof by
   contradiction is condition (b) and (c) of fairness.

  \item The fact that a goal  may have different names i.e. we may have
   $F(n_i)=F(n_j)$ for $i \neq j$ comes from the following reason.
   A goal $h$ may appear in different  proofs of $g$ or in  different parts of a proof of $g$. Of course
  $\di{(h)}$ does not depend on the place where
   $h$ appears but the condition that lets us decide to give it a
   new name or not depends of this place. We know, by lemma
   \ref{27}, that $\di(h)=\di(h')$ if $h \sim h'$ and $h$ is below $h'$
  in some part of a proof but there is no reason to have
  $\di(h)=\di(h')$ if they appear in different proofs or in independent part of
  a
  proof.
\end{enumerate}

\begin{lem}\label{fin}
The algorithm given above terminates.
\end{lem}
\begin{proof}
Since the goals are made of sub-formulae of the formulae in $g$,
there are only finitely many possible non equivalent goals. Also
note that, when we try to find a proof for a goal $h$ and we have
to consider some goal $h_1$, we  give a new name to $h_1$ (i.e. we
introduce a new variable $n_i$ such that $F(n_i)=h_1$ for which,
later, we will have to find an equation) only when there is no
$h_2 \sim h_1$ below $h$ in the branch of the proof of $g$ that
the algorithm, intuitively, constructs. Thus, all the branches are
finite. Since there are only finitely many rules, by Konig's
lemma, only finitely many variables can be introduced and thus the
algorithm terminates.
\end{proof}

\subsection{Proof of  theorem \ref{thm}}

It  is an immediate consequence of lemma \ref{27} and theorem
\ref{main} below.

\begin{thm}\label{main}
Let $g$ be a goal and let $\Vec{a}$ be the minimal solution of
$PSE(g)$. Then, for each variable $n_i$ occurring in $PSE(g)$ we
have $a_i=\di(F(a_i))$.
\end{thm}
\begin{proof}
 Let $PSE(g)$ be the set $\Vec{n} =P(\Vec{n})$ of equations and $\Vec{b}$ be defined by $b_i=\di(F(n_i))$.
  It follows from lemma \ref{27}
 that $\Vec{b}$ is a solution of $PSE(g)$.
Thus, we have $\Vec{a} \leq \Vec{b}$. Let $u_k=P^k(\Vec{0})$.
Since $\Vec{a}$ is the minimal solution of the system $\Vec{n}
=P(\Vec{n})$ we have $\Vec{a}=\bigsqcup_{k=0}^{\infty} u_k$.
Denote by $d_i(k)$ the number of normal and fair proofs of
$F(n_i)$ of size $k$ and $\overrightarrow{d(k)}$ the vector whose
components are the $d_i(k)$. Then
$\Vec{b}=\bigsqcup_{k=0}^{\infty} \overrightarrow{d(k)}$. Note
that the equations are done so that $\overrightarrow{d(k+1)}\leq
P(\overrightarrow{d(k)})$. An immediate induction shows that, for
each $k$, $\overrightarrow{d(k)} \leq u_k$. It follows then that
$\Vec{b} \leq \Vec{a}$.
\end{proof}

\noindent {\bf Remark}

If, instead of interpreting the variables and coefficients in
$\NA$,  we interpret them in the set $\{0,1\}$ where the
operations and the order are the ones of $\N$ except that $1+1=1$,
the conclusion of the theorem is then that $a_h=1$ iff the goal
$h$ is provable.

\subsection{Some other restrictions on proofs}\label{rest}

\begin{defn}
We say that a normal term $t$ is in $\eta$-long normal form if the
following holds for every sub-term $u$ of $t$.
\begin{enumerate}
  \item If $u$ has type $A \f B$ then either $u=\l x. u'$ or $u$ is
  applied to some other term.
  \item If $u$ has type $A \etrd B$ then $u=\< u_1,u_2 \>$ for some
  terms $u_1,u_2$.
\end{enumerate}
\end{defn}

\medskip

The algorithm we have given in the previous sections has been
designed to get the number of normal and fair proofs in classical
logic. It can be easily transformed if we want to only count
proofs satisfying some constraints.
\begin{enumerate}
  \item If we want to have proofs in minimal logic  i.e.
  the logic where the rules $\bot_i$ and $ \bot_e$ are deleted, we just forget
  the second step (which corresponds to proof by contradiction) in
  the definition of the set of equations
  \item If we want to have proofs in intuitionistic logic, i.e.
  the logic where the rules $\bot_i,$ and $ \bot_e$ are deleted and
  replaced by the rule
    $$\F{\G \v  \bot} {\G \v A }$$
 we replace the polynomial given in the second step of
  the definition of the set of equations by the following one. If $g$
  is $\G \v B$ and $B \neq \bot$ then $Q=n_h$ where $h$ is $\G \v
  \bot$ and $Q=0$ otherwise.
  \item Instead of changing the logic, we may also want to
  restrict the form of the proofs we are looking for. The main
  usual restriction is to ask to have proofs in $\eta$-long normal
form. It is well known that, with this restriction, the system
remains complete. If we want such proofs it is enough, in the
definition of the equations to ask that, if the goal is $\G \v B$
and the main connector of $B$ is either an arrow or a conjunction,
then we cannot use a proof by contradiction or use an elimination
rule.
\item Our algorithm gives two normal and fair proofs for the formula $A \f A$.
These proofs are $\l x. x$ and $\l x. \m\a. (\a \ x)$. We could
consider that these two proofs are the same and, actually, there
is a reduction rule in the $\l\m$-calculus that ensures that the
second term reduces to the first one. This rule, that looks like
the $\eta$-rule of the $\l$-calculus, is the following $\m\a. (\a
\ M) \trrd M$ if $\a$ does not occur in $M$. It intuitively means
that if, in a  proof of $A$  by contradiction, in fact you have a
proof $M$  of $A$  that does not use $\neg A$, you can eliminate
the use of the rule for  proof by contradiction.

It would be more difficult to consider this rule in the definition
of normal proof. This is because it is non local and our
algorithm, by essence, can only consider local configurations.
\end{enumerate}

\subsection{From polynomials to formulae}

 In the previous sections we have associated to each formula $F$  a set of polynomial
  equations whose minimal solution gives the number of normal and fair proofs
  of $F$.
 The opposite construction is also possible as the next proposition
 shows.

\begin{defn}
Let $F$ be a formula of implicational propositional logic i.e. $F$
is built from atomic formulae by using only the arrow as
connectors. The rank of $F$ (denoted as $r(F)$) is defined by the
following rules.
\begin{enumerate}
  \item If $F$ is atomic, then $r(F)=0$
  \item If $F=A \f B$, then $r(F)=max(r(A)+1, r(B))$

\end{enumerate}

\end{defn}

 \bt{prop}
 Let $E$ be a polynomial system of
equations with $n$ variables. We can compute  $n$ formulae $A_1,
..., A_n$  of implicational logic  such that, if $(a_1, ..., a_n)$
is  the minimal solution  of $E$ then, for all $i \leq n$, $a_i$
is the number of proofs of $A_i$ in $\eta$-long normal form.
Moreover we may assume that $r(A_i) \leq 2$ for all $i \leq n$.
\et{prop}

\begin{proof}
 Let $\Vec{x}=F(\Vec{x})$ be the system and $F=(f_1 ,...,f_n )$.
 We take $n$ fresh ground types $O_1, \ldots ,O_n$.
 For each polynomial $f_p$ we construct a formula $B_p$ in the following
 way.
 For each monomial $M_i = x_1^{\alpha_1} \cdot \ldots \cdot x_n^{\alpha_n}$
 which  appears  in $f_p$ let
 $T_i$ be the formula  $O_1^{\alpha_1} ,\ldots , O_n^{\alpha_n}  \ar O_p$.
 Remember
 that constant $1$ can be obtained as the monomial
 $x_1^{\alpha_1} \cdot \ldots \cdot x_n^{\alpha_n}$ when all $\alpha_i =0$.
 The formula associated to $f_p$ is  $T_1 ,\ldots, T_{m} \ar O_p$.
 The fact that these formulae satisfy the desired conclusion is straightforward.
\end{proof}

\section{Examples}

\noindent {\bf Example 1}

We want to compute the number of normal and fair proofs of the
formula $F$ below

 $$ F= F_1 \ar F_2 \ar F_3 \ar F_4 \ar F_5 \ar F_6 \ar A$$

\noindent where

$F_1= B \ar C \ar C$ \hspace{1cm} $F_2=F_3=C$
 \hspace{1cm} $F_4=B \ar C \ar B$

 $F_5=C \ar
C \ar A$ \hspace{1cm} $F_6=A \ar B \ar A$

\medskip

To avoid too many equations we will restrict ourselves to proofs
in $\eta$-long normal form and in minimal logic and, to simplify
notations, we will use the same name for a goal and the variable
attached to it and, if a goal has several names, the corresponding
variables will be the same with, possibly, some index. Also note
that, since we will not write the terms representing the proofs,
there is no need to give names to the hypothesis and thus we will
write contexts simply as multi-sets of formulae.

\noindent Let $\G= F_1, F_2, F_3,  F_4,  F_5,  F_6$.
 The goals are:

\begin{center}
 $x $ is $ \G \v A$,

 $y,y_1 $ are $\G \v B$

 $z,z_1 $ are $ \G \v C$.

\end{center}

 \noindent  The order on
these variables is given by: $x<y,z$ ;  $y<z_1$ and $z<y_1$.

\noindent The set of equations is
\medskip

\begin{center}

\noindent $x=xy+z^2$\\
$y=yz_1$ \hspace{2cm} $z_1=2+yz_1$ \\
$z=2+y_1z$ \hspace{2cm} $y_1=y_1z$

\end{center}

\medskip

\noindent The minimal solution is $x=4, y=y_1=0, z=z_1=2$ and,
therefore, there are exactly 4 proofs of $F$ in
 $\eta$-long normal form.

 \medskip

\noindent {\bf Example 2}

 We want to compute the number of normal and fair proofs of the formula $F$
 below
 where $\neg_c B$ is the abbreviation of $B \ar C$.
 This formula is a kind of translation (provable in minimal logic)
 of Pierce law.

\medskip

 $$ F= ((A \ar \neg_c\neg_c B) \ar  \neg_c \neg_c A) \ar \neg_c \neg_c
 A$$

\medskip

Again, we adopt the same restrictions and conventions of notations
as in the previous example.

 \noindent Let $F_1= (A \ar \neg_c\neg_c B) \ar \neg_c \neg_c A$, $F_2=
 \neg_c A$ and $\G=\a_1 : F_1, \a_2 : F_2$.

 \noindent The goals are

 \begin{center}

 $x$ is
 $\G \v C$,

 $y, y_1$ are
 $\G,  A,  \neg_c B \v C$,

 $z, z_1$ are  $\G,  A \v
 C$,

 $u$ is $\G \v A$,

 $v, v_1$ are
 $\G,  A,  \neg_c B \v A$,

  $w, w_1$  are
 $\G,  A, \neg_c B \v B$

$r,r_1$ are $\G,  A \v
 A$.

 \end{center}

\noindent The order on these variables is given by: $x<y,z,u$ ;
$y<z_1,v,w$ ;  $z_1<r_1$ ; $z<y_1, r$ ; $y_1 <v_1, w_1$

\noindent  The set of  equations is

\medskip

\begin{center}

$x=yz+u$

  $y=yz_1+v+w $ \hspace{1cm} $z_1=yz_1+r_1$ \hspace{1cm}

$z=y_1z+r$ \hspace{1cm} $y_1=y_1z+v_1+w_1 $

  $v=v_1=1$ \hspace{1cm}
  $w=w_1=0$ \hspace{1cm}  $r=r_1=1$ \hspace{1cm} $u=0$

\end{center}

\medskip

\noindent The minimal solution is $x=y=z=y_1=z_1=\omega$ and,
therefore, there are infinitely many proofs of $F$ in $\eta$-long
normal forms.

 \medskip

\noindent {\bf Example 3}

Let $F$ be the formula $\neg A \ou A$. It is known that $F$ is not
provable in intuitionistic logic.  We will show that, in classical
logic, the are infinitely many distinct proofs in $\eta$-long
normal form. Since the number of equations to be written is quite
big we will only write some of those that imply that the number is
infinite. To simplify we will also omit some intermediate goals
and/or equations when the relations between the corresponding
variables are easy to show.

The useful goals are the following

\begin{center}
$x$ is $\v F$

$x_1$ is $\v A$, $x_2$ is $\v \neg A$ and $x_3$ is $\a : \neg F \v
\bot$

$a$ is $\a : \neg F \v A$ and $b$ is $\a : \neg F \v \neg A$

$a_1$ is $\a : \neg F, \b : \neg A \v \bot$ and $a_2$ is $\a :
\neg F \v \neg A$

$c$ is $\a : \neg F, \b : \neg A, y:A  \v \bot$

$c_1$ is $\a : \neg F, \b : \neg A, y:A  \v A$

$d$ is $\a : \neg F, \b : \neg A, y:A , z:A \v \bot$

\end{center}

Some equations are

\begin{center}
$x=x_1+x_2+x_3$

 $x_1=0$, $x_2=0$

 $x_3=a+b$

 $a=a_1+a_2$

 $a_1=c$ \ \ \ ($\star$)

 $c=2.c_1+d$

 $c_1=1$

 The use of lemma \ref{27} gives $d=c$.
\end{center}

\noindent ($\star$) $a_1$ actually is the sum of $c$ and some
 other variables that are easily shown to be 0.

\end{document}